\documentclass[12pt]{article}
\usepackage{geometry} 
\usepackage[ utf8 ]{inputenc}
\usepackage[T1]{fontenc}
\usepackage[english,french]{babel}
\usepackage{amsmath}
\usepackage{amsfonts}
\usepackage{amssymb}
\usepackage{amsthm}
\usepackage{textcomp}
\usepackage{eucal}
\usepackage{array}
\usepackage{pb-diagram}
\theoremstyle{plain}
\newtheorem{theorem}{Théorème}
\newtheorem{corollary}{Corollaire}

\newtheorem{proposition}{Proposition}
\newtheorem{notation}{Notation}
\theoremstyle{definition}
\newtheorem{definition}{Définition}

\theoremstyle{remark}
\newtheorem{remark}{Remarque}
\date{}

\title{ Représentations de réflexion de groupes de Coxeter\\Cinquième partie: La représentation $R$ est réductible. Cas particulier du rang $3$}
\author{François ZARA
}
\begin{document}
\maketitle
\begin{abstract}
Dans cette cinquième partie, (avec les notations des parties précédentes) on fait les hypothèses suivantes: $(W,S)$ est un système de Coxeter irréductible, $2$-sphérique et $S$ est de cardinal $3$. Soit $R:W\to GL(M)$ une représentation de réflexion réductible de $W$. On pose $G:= Im R$. Chaque sous-espace de $M$ $(\neq M)$ fixé par $G$ est contenu dans $C_{M}(G)$. On pose $M':=M/C_{M}(G)$ et $N(G):=\{g|g\in G, g \,\text{fixe}\, M'\}$. On appelle $N(G)$ le sous-groupe des translations de $G$. Un des buts de cette partie est d'étudier $M'$ et $N(G)$. En particulier, on montre que $N(G)$ a une structure de $\mathcal{O}G-module$.
\end{abstract}
\begin{otherlanguage}{english}
\begin{abstract}
In this fith part, (with the notations of the preceding parts) we make the following hypothesis:  $(W,S)$ is a  Coxeter system, irreducible, $2$-spherical and $S$ is of cardinality $3$. Let $R:W\to GL(M)$ be a reducible reflection representation of $W$. Let $G:= Im\,R$. Each sub-space of $M$ $(\neq M)$ stabilize by $G$ is contained in $C_{M}(G)$. Let $M':=M/C_{M}(G)$ and $N(G):=\{g|g\in G,g\, \text{acts trivially on}\,M'\}$. We call $N(G)$ the translation sub-group of $G$. One of the goals of this part is to study $M'$ and $N(G)$. In particular it is shown that $N(G)$ is an $\mathcal{O}G-module$.
\end{abstract}
\end{otherlanguage}
\let\thefootnote\relax\footnote{Mots clés et phrases: groupes de Coxeter, groupes de réflexion.Représentation de réflexion réductible.}
\let\thefootnote\relax\section{Introduction}
Nous gardons les hypothèses et notations des parties précédentes mais en plus nous supposons que le rang est $3$. Avec cette hypothèse supplémentaire de nombreuses simplifications se produisent. 

Dans toute la suite $(W,S)$ est un système de Coxeter de rang $3$ avec $W=W(p,q,r)$ et $S=\{s_{1},s_{2},s_{3}\}$, et la représentation $R$ est réductible (sauf mention du contraire). On a posé $G:=Im R$. Si $\mathcal{O}'(K)$ est le sous-anneau des entiers algébriques de $K$ engendré par les éléments de $\mathcal{P}(G)$ alors on montre que $N(G)$ est un $\mathcal{O}'(K)(G)$-module.\\
Si $2|pqr$, on montre qu'une condition nécéssaire pour que la suite $(\star\star)$ soit scindée est que l'on puisse trouver des éléments $\lambda_{i}\in K$ tels que la condition $(\mathcal{E})$ soit satisfaite:
\[
(\mathcal{E}) \quad (4-\gamma)\lambda_{1}+(l+2)\lambda_{2}+(m+2)\lambda_{3}=-1.
\]
On montre ensuite que l'on peut étendre $G$ en ajoutant une réflexion bien choisie à l'aide de certains éléments de $\mathcal{Z}'$.
\section{Des relations entre les paramètres.}
On pose, si $\{i,j,k\}=\{1,2,3\}$, $<b_{k}>=H(s_{i})\bigcap H(s_{j})$. Alors, à facteur de proportionnalité près, nous obtenons:
\begin{eqnarray}
b_{1} & = & (4-\gamma)a_{1}+(l+2)a_{2}+(m+2)a_{3} \\
b_{2} & = & (2\alpha+\beta m)a_{1}+(4-\beta)a_{2}+(\alpha+2m)a_{3}\\
b_{3} & = & (\alpha l+2\beta)a_{1}+(\beta+2l)a_{3}
\end{eqnarray}
Dans la suite on pose $b=b_{1}$.

On a $\det(b_{1},b_{2},b_{3})=\Delta(G)^{2}$, donc si $\Delta(G)\neq 0$, $(b_{1},b_{2},b_{3})$ est une base de $M$ et si $\Delta(G)=0$ on a $<b_{1}>=<b_{2}>=<b_{3}>=C_{M}(G)$.

Si $x=\sum_{1}^{3}\lambda_{i}a_{i}$ est un élément de $M-\{0\}$, on a:
\begin{eqnarray*}
s_{1}(x) & = & x-(2\lambda_{1}-\alpha\lambda_{2}-\beta\lambda_{3})a_{1}\\
s_{2}(x) & = & x-(-\lambda_{1}+2\lambda_{2}-l\lambda_{3})a_{2}\\
s_{3}(x) & = & x-(-\lambda_{1}-m\lambda_{2}+2\lambda_{3})a_{3}
\end{eqnarray*}
donc $x \in C_{M}(G)$ si et seulement si le système linéaire homogène $(S)$
\begin{eqnarray*}
2X_{1}-\alpha X_{2}-\beta X_{3} & = & 0\\
-X_{1}+2X_{2}-lX_{3} & = & 0\\
-X_{1}-mX_{2}+2X_{3} & = & 0
\end{eqnarray*}
a $(\lambda_{1},\lambda_{2},\lambda_{3})$ comme solution. Le déterminant de $(S)$ est 
\[
\Delta=8-2\alpha-2\beta-2\gamma-(\alpha l+\beta m).
\]
Il en résulte que si $\Delta \neq 0$ la représentation $R(\alpha,\beta,\gamma;l)$ est irréductible et si $\Delta=0$, on a $<b_{1}>=<b_{2}>=<b_{3}>$ et le système $(S)$ est de rang $1$.

Soit \[T:=\begin{pmatrix}
4-\gamma & 2\alpha+\beta m & 2\beta+\alpha m\\
l+2 & 4-\beta & \beta+2l\\
m+2 & \alpha+2m & 4-\alpha
\end{pmatrix}.\]
- Si $\Delta\neq 0$ on a:
\[
T^{-1}=\frac{1}{\Delta}\begin{pmatrix}
2 & -\alpha & -\beta\\
-1 & 2 & -l\\
-1 & -m & 2
\end{pmatrix}
=\frac{1}{\Delta}C(R(\alpha,\beta,\gamma;l))
\]
et un calcul simple montre que
\[
s_{i}(b_{i})=b_{i}-\Delta a_{i}\quad (1\leqslant i \leqslant 3)
\]
et 
\begin{eqnarray*}
\Delta a_{1} & = & 2b_{1}-b_{2}-b_{3}\\
\Delta a_{2} & = & -\alpha b_{1}+2b_{2}-mb_{3}\\
\Delta a_{3} & = & -\beta b_{1}-lb_{2}+2b_{3}
\end{eqnarray*}
- Si $\Delta=0$, tous les mineurs d'ordre $2$ de $T$ sont nuls. Nous en déduisons les neuf relations $(\mathcal{T})$:
\begin{eqnarray*}
(4-\alpha)(2\alpha+\beta m) & = & (\alpha+2m)(2\beta+\alpha l)\\
(4-\beta)(2\beta+\alpha m) & = & (\beta+2l)(2\alpha+\beta m)\\
(4-\alpha)(4-\beta) & = & (\beta+2l)(\alpha+2m)\\
(4-\alpha)(l+2) & = & (m+2)(\beta+2l)\\
(4-\gamma)(\beta+2l) & = & (l+2)(2\beta+\alpha l)\\
(4-\alpha)(4-\gamma) & = & (m+2)(2\beta+\alpha l)\\
(4-\beta)(m+2) & = & (l+2)(\alpha+2m)\\
(4-\gamma)(\alpha+2m) & = & (m+2)(2\alpha+\beta m)\\
(4-\beta)(4-\gamma) & = & (l+2)(2\alpha+\beta m)
\end{eqnarray*}
et chacune de ces relations équivaut à $\Delta=0$.

Comme $0<\alpha,\beta,\gamma<4$, nous voyons que $l+2\neq 0 \neq m+2$, $\alpha+2m\neq 0 \neq \beta+2l$ et $2\alpha+\beta m \neq 0 \neq 2\beta+\alpha l$.
\section{Etude du groupe $N(G)$.}
Soit $W(p,q,r)$ un groupe de Coxeter de rang $3$ et soit $R$ une de ses représentations de réflexion. On appelle $\mathcal{P}(G)$ son système de paramètres: $\mathcal{P}(G)=\mathcal{P}(\alpha,\beta,\gamma;\alpha l,\beta m)$ et on pose : $G=Im R$.

On a déjà vu que $R$ est réductible si et seulement si 
\[\Delta = 8-2\alpha-2\beta-2\gamma -(\alpha l+\beta m)=0\] \textbf{ce que l'on suppose dans toute la suite}. On en déduit que dans ce cas, $\alpha l$ et $\beta m$ sont les racines  du polynôme $Q(X)\in K_{0}(X)$:
\[
Q(X):=X^{2}-2(4-\alpha-\beta-\gamma)X+\alpha\beta\gamma.
\]
donc, en posant $\theta:=-4+\alpha+\beta+\gamma+\alpha l$:
\[
\{\alpha l, \beta m\}=\{4-\alpha-\beta-\gamma+\sqrt{\theta},4-\alpha-\beta-\gamma-\sqrt{\theta}\}
\]
Il en résulte que $K$ est un corps de décomposition de $Q(X)$.

Nous avons déjà étudié dans la partie 3 le cas où $Q(X)$ possède une racine double: on obtient les groupes diédraux affines.\\
On suppose maintenant que $Q(X)$ possède deux racines distinctes. Soit $\sigma$ un générateur du groupe de Galois $\mathcal{G}(K/K_{0})$. Alors $\sigma(\alpha l)=\beta m$ et on obtient la représentation conjuguée en appliquant $\sigma$.

Nous supposons maintenant que $N(G) (=N)$ est non trivial. Nous avons déjà vu dans la partie  4 que $N$ est un groupe commutatif sans torsion. Nous allons ici préciser sa structure.

Nous avons la base $\mathcal{B}=(b,a_{2},a_{3})$ de $M$. En effet, si $\mathcal{B}$ est un système lié on a $b\in<a_{2},a_{3}>$, donc $\gamma=4$ ce qui n'est pas.

Si $\zeta\in N$, la matrice de $\zeta$ dans la base $\mathcal{B}$ est de la forme:
\[
\zeta=
\begin{pmatrix}
1 & \lambda & \mu\\
0 & 1 & 0\\
0 & 0 & 1
\end{pmatrix}
\]
où $\lambda,\mu$ sont dans $K$. Dans ces conditions nous écrirons $\zeta:=(\lambda,\mu)$: $\zeta$ peut être considéré comme un élément de l'espace vectoriel $K^{2}$.

Si $\zeta_{i}=(\lambda_{i},\mu_{i})$ $(i=(1,2))$ sont dans $N$, alors $\zeta_{1}\zeta_{2}=(\lambda_{1}+\lambda_{2},\mu_{1}+\mu_{2})$, donc $N$ est isomorphe à un sous-groupe de $K^{2}$, en particulier $N$ est sans torsion. Comme $N$ est un sous-groupe normal de $G$, celui-ci opère sur lui par automorphismes intérieurs:\\
Soit $h\in G$ et $\zeta=(\lambda,\mu)\in N$. Nous définissons $h.\zeta$ par $h.\zeta=(\lambda',\mu')$ si\\ $h\zeta h^{-1}=\begin{pmatrix}
1 & \lambda' & \mu'\\
0 & 1 & 0\\
0 & 0 & 1
\end{pmatrix}
.$
\begin{notation}
On appelle $\mathcal{O}'(K)$ le sous-anneau de $\mathcal{O}(K)$ engendré par $\alpha,\beta,\gamma,\alpha l,\beta m$. C'est un sous-anneau de l'anneau des entiers algébriques de $K$.
\end{notation}
On peut remarquer que si $\alpha$ et $\beta$ sont inversibles, alors $\mathcal{O}'(K)=\mathcal{O}(K)$.
\begin{theorem}
$N$ est un $ \mathcal{O}'(K)$-module et aussi $N$ est stable par multiplication par $\theta$.
\end{theorem}
\begin{proof}
On montre que si $\zeta\in N$ alors $\alpha\zeta,\beta\zeta,\gamma\zeta,\alpha l\zeta$ sont dans $N$. Soit $\zeta=(\lambda,\mu)\in N$. Un calcul facile montre que:
\[
\zeta(b)=b,\zeta(a_{1})=a_{1}+\omega(\zeta)b,\zeta(a_{2})=a_{2}+\lambda b,\zeta(a_{3})=a_{3}+\mu b
\]
où $\omega(\zeta)=-\frac{\lambda(l+2)+\mu(m+2)}{4-\gamma}$.\\
Nous obtenons:
\[
s_{1}.\zeta=\zeta+\omega(\zeta)(\alpha,\beta),s_{2}.\zeta=\zeta+\lambda(-2,l),s_{3}.\zeta=\zeta+\mu(m,-2).
\]
Nous posons: $c_{1}:=(\alpha,\beta)$, $c_{2}:=(-2,l)$ et $c_{3}:=(m,-2)$. Nous vérifions facilement, en utilisant les relations $\mathcal{T}$, que si $i\neq j$, $(c_{i},c_{j})$ est une base de $K^{2}$.\\
Nous définissons maintenant des applications linéaires de $K^{2}$ dans lui même. On pose $\psi_{i}(\zeta):=[s_{i},\zeta]$ $(1\leqslant i \leqslant3)$, $\psi_{ij}(\zeta):=[s_{i},\psi_{j}(\zeta)]$ $(1\leqslant i \neq j \leqslant3)$ et $\psi_{ijk}(\zeta):=[s_{i},\psi_{jk}(\zeta)]$ $(1\leqslant i\neq j\neq k\leqslant3)$ si $\zeta=(\lambda,\mu)\in K^{2}$.\\
Nous pouvons remarquer que les applications $\psi_{i}$, $\psi_{ij}$, et $\psi_{ijk}$ stabilisent $N$.

Nous avons les formules dont la vérification n'offre pas de difficultés:
\[
\psi_{1}(\zeta)=\omega(\zeta)c_{1},\psi_{2}(\zeta)=\lambda c_{2},\psi_{3}(\zeta)=\mu c_{3};
\]
\begin{multline*}
\psi_{12}(\zeta)=\lambda c_{1},\psi_{13}(\zeta)=\mu c_{1},\psi_{21}(\zeta)=\omega(\zeta)\alpha c_{2},\psi_{23}(\zeta)=\mu mc_{2},\\ \psi_{31}(\zeta)=\omega(\zeta)\beta c_{3},\psi_{32}(\zeta)=\lambda lc_{3};
\end{multline*}
\begin{multline*}
\psi_{123}(\zeta)=\mu mc_{1},\psi_{132}(\zeta)=\lambda lc_{1},\psi_{213}(\zeta)=\mu \alpha c_{2},\psi_{231}(\zeta)=\omega(\zeta)\beta mc_{2}\\
\psi_{312}(\zeta)=\lambda\beta c_{3},\psi_{321}(\zeta)=\omega(\zeta)\alpha lc_{3}.
\end{multline*}
Nous avons $(\psi_{3}-\psi_{2})(\zeta)=-\lambda c_{2}+\mu c_{3}$ donc $(\psi_{3}-\psi_{2})(c_{2})=2c_{2}+lc_{3}$ et $(\psi_{3}-\psi_{2})(c_{3})=-mc_{2}-2c_{3}$. Un calcul simple montre que $(\psi_{3}-\psi_{2})^{2}(c_{i})=(4-\gamma)c_{i}$ $(i\in\{2,3\})$. Comme $(c_{2},c_{3})$ est un base de $K^{2}$, nous voyons que $(\psi_{3}-\psi_{2})^{2}$ est l'homothétie de rapport $4-\gamma$ sur $K^{2}$. En particulier si $\zeta\in N$, alors $(\psi_{3}-\psi_{2})^{2}(\zeta)=(4-\gamma)\zeta\in N$. Comme $4\zeta\in N$ nous obtenons $\gamma\zeta\in N$.\\
Nous voyons de même que $(\psi_{1}-\psi_{3})^{2}$ (resp. $(\psi_{2}-\psi_{1})^{2}$) est l'homothétie de rapport $4-\beta$ (resp. $4-\alpha$), d'où comme ci-dessus $\beta\zeta$ et $\alpha\zeta$ sont dans $N$ si $\zeta$ est dans $N$.\\
Comme $(c_{1},c_{2})$ est une base de $K^{2}$, nous pouvons écrire:
\[
\zeta=(\lambda,\mu)=\Big(\frac{l\lambda+2\mu}{\alpha l+2\beta}\Big)c_{1}+\Big(\frac{-\beta\lambda+\alpha\mu}{\alpha l+2\beta}\Big)c_{2}
\]
d'où $(\alpha l+2\beta)\zeta=(l\lambda+2\mu)c_{1}+(-\beta\lambda+\alpha\mu)c_{2}$. Nous montrons que le membre de droite de cette égalité est dans $N$ si $\zeta$ est dans $N$.
\begin{itemize}
\item Nous avons $\psi_{132}(\zeta)=l\lambda c_{1}$, donc $l\lambda c_{1}\in N$;
\item nous avons $\psi_{13}(\zeta)=\mu c_{1}$, donc $2\mu c_{1}\in N$;
\item nous avons $\psi_{1}(\zeta)=\lambda c_{2}$, donc $ -\beta\lambda c_{2}\in N$;
\item nous avons $\psi_{213}(\zeta)=\alpha\mu c_{2}$, donc $\alpha\mu c_{2}\in N$.
\end{itemize}
Il en résulte que $(\alpha l+2\beta)\zeta\in N$. Comme $2\beta\zeta\in N$, nous obtenons $\alpha l\zeta\in N$, donc $N$ est un $\mathcal{O}'(K)$-module.\\
Comme $\alpha l=4-\alpha-\beta-\gamma+\theta$, nous voyons que $\theta\zeta\in N$.
\end{proof}
Les formules données dans la démonstration de théorème 1: si $\zeta=(\lambda,\mu)\in K^{2}$, $s_{1}.\zeta=\zeta+\omega(\zeta)c_{1}$, $s_{2}.\zeta=\zeta+\lambda c_{2}$, $s_{3}.\zeta=\zeta+\mu c_{3}$ définissent une action de $G$ sur $K^{2}$ et l'on voit par un calcul simple que l'on a les relations:
\[
\begin{array}{c|ccc}
& s_{1}.c_{1} & = & -c_{1}\\
& s_{1}.c_{2} & = & c_{2}+c_{1}\\
& s_{1}.c_{3} & = & c_{3}+c_{1}
\end{array}
,
\begin{array}{c|ccc}
& s_{2}.c_{1} & = & c_{1}+\alpha c_{2}\\
& s_{2}.c_{2} & = & -c_{2}\\
& s_{2}.c_{3} & = & c_{3}+mc_{2}
\end{array}
,
\begin{array}{c|ccc}
& s_{3}.c_{1} & = & c_{1}+\beta c_{3}\\
& s_{3}.c_{2} & = & c_{2}+lc_{3}\\
& s_{3}.c_{3} & = & -c_{3}
\end{array}
\]
et aussi $(4-\gamma)c_{1}+(2\alpha+\beta m)c_{2}+(2\beta+\alpha l)c_{3}=0$, et les relations équivalentes: $(m+2)c_{1}+(\alpha+2m)c_{2}+(4-\alpha)c_{3}=0$ et $(l+2)c_{1}+(4-\beta)c_{2}+(\beta+2l)c_{3}=0$ qui résultent immédiatement des relations $(\mathcal{T})$.
\begin{proposition}
L'opération précédente de $G$ sur $K^{2}$ est une représentation de réflexion de $W(p,q,r)$. De plus:\\
1) -Si $Q(X)$ admet une racine double, cette opération est isomorphe à l'action de $G$ sur $M'$.\\
    -Si $Q(X)$ admet deux racines distinctes $e$ et $f$, alors on a la représentation $R(\alpha,\beta,\gamma;\alpha^{-1}e)$ sur $M'$ et la représentation $R(\alpha,\beta,\gamma;\alpha^{-1}f)$ sur $K^{2}$.\\
    Dans tous les cas la représentation de $G$ sur $K^{2}$ est irréductible.\\
2) Si $Q(X)$ admet deux racines distinctes dans $K-K_{0}$, les représentations $R(\alpha,\beta,\gamma;\alpha^{-1}e)$ et $R(\alpha,\beta,\gamma;\alpha^{-1}f)$ sont conjuguées.
\end{proposition}
\begin{proof}
Posons, pour cette démonstration uniquement, $d_{1}:=c_{1}$, $d_{2}:=\alpha c_{2}$, et $d_{3}:=\beta c_{3}$. Nos avons les relations:
\[
\begin{array}{c|ccc}
& s_{1}.d_{1} & = & -d_{1}\\
& s_{1}.d_{2} & = & d_{2}+\alpha d_{1}\\
& s_{1}.d_{3} & = & d_{3}+\beta d_{1}
\end{array}
,
\begin{array}{c|ccc}
& s_{2}.d_{1} & = & d_{1}+ d_{2}\\
& s_{2}.d_{2} & = & -d_{2}\\
& s_{2}.d_{3} & = & d_{3}+\alpha^{-1}\beta md_{2}
\end{array}
,
\begin{array}{c|ccc}
& s_{3}.d_{1} & = & d_{1}+ d_{3}\\
& s_{3}.d_{2} & = & d_{2}+\beta^{-1}\alpha ld_{3}\\
& s_{3}.d_{3} & = & -d_{3}
\end{array}
.
\]
Si $Q(X)$ admet une racine double, nous avons $\alpha l=\beta m$ donc $s_{2}.d_{3}=d_{3}+ld_{2}$ et $s_{3}.d_{2}=d_{2}+md_{3}$. L'opération considérée est isomorphe à l'action de $G$ sur $M'$.\\
Si $Q(X)$ admet deux racines distinctes $e$ et $f$, alors si l'on pose $l':=\alpha^{1}\beta m$ et $m':=\beta^{-1}\alpha m$, nous obtenons $\alpha l'=\beta m$ et $\beta m'=\alpha l$, d'où le résultat dans ce cas. Si $e$ et $f$ ne sont pas dans $K_{0}$, alors nous avons $\sigma(\alpha l)=\beta m=\alpha\sigma(l)$ et nous obtenons $\sigma(l)=\alpha^{-1}\beta m=l'$ et de même $\sigma(m)=m'$ donc $s_{2}.d_{3}=d_{3}+\sigma(l)d_{2}$, $s_{3}.d_{2}=d_{2}+\sigma(m)d_{3}$. les représentations $R(\alpha,\beta,\gamma;\alpha^{-1}e)$ et $R(\alpha,\beta,\gamma;\alpha^{-1}f)$ sont conjuguées. Il est clair que dans tous les cas, la représentation de $G$ sur $K^{2}$ est irréductible.
\end{proof}
\section{Un résultat sur l'extension de $G/N$ par $N$.}
Dans ce paragraphe nous donnons une condition nécéssaire pour que la suite exacte:
\[
(**)\quad
\{1\}\longrightarrow N\longrightarrow G\longrightarrow G'\longrightarrow\{1\}
\]
soit scindée.

Pour les calculs sur l'extension de $G'$ par $N$, nous aurons besoin des résultats préliminaires suivants:
\begin{proposition}
On garde les hypothèses précédentes.\\
\begin{enumerate}
  \item On pose $g_{1}:=s_{3}s_{2}$. Dans la base $(c_{2},c_{3})$ de $K^{2}$ on a:
  \[
  \forall n\in N,g_{1}^{n}=
  \begin{pmatrix}
-u_{2n-1}(\gamma) & mu_{2n}(\gamma)\\
-lu_{2n}(\gamma) & u_{2n+1}(\gamma)
\end{pmatrix}
  \]
  \begin{itemize}
  \item Si $n$ est pair on a:
  \[
  Id+g_{1}+g_{1}^{2}+\cdots+g_{1}^{n-1}=u_{n}(\gamma)
  \begin{pmatrix}
-\gamma u_{n-2}(\gamma) & mu_{n-1}(\gamma)\\
-lu_{n-1}(\gamma) & \gamma u_{n}(\gamma)
\end{pmatrix}
  \]
  \item 
  Si $n$ est impair on a:
  \[
  Id+g_{1}+g_{1}^{2}+\cdots+g_{1}^{n-1}=u_{n}(\gamma)
  \begin{pmatrix}
- u_{n-2}(\gamma) & mu_{n-1}(\gamma)\\
-lu_{n-1}(\gamma) & u_{n}(\gamma)
\end{pmatrix}
  \]
\end{itemize}
  \item 
  On pose $g_{2}:=s_{1}s_{3}$. Dans la base $(c_{1},c_{3})$ de $K^{2}$ on a:
  \[
  \forall n\in N,g_{2}^{n}=
  \begin{pmatrix}
-u_{2n-1}(\beta) & u_{2n}(\beta)\\
-\beta u_{2n}(\beta) & u_{2n+1}(\beta)
\end{pmatrix}
  \]
  \begin{itemize}
  \item Si $n$ est pair on a:
  \[
  Id+g_{2}+g_{2}^{2}+\cdots+g_{2}^{n-1}=u_{n}(\beta)
  \begin{pmatrix}
-\beta u_{n-2}(\beta) & u_{n-1}(\beta)\\
-\beta u_{n-1}(\beta) & \beta u_{n}(\beta)
\end{pmatrix}
  \]
  \item 
  Si $n$ est impair on a:
  \[
  Id+g_{2}+g_{2}^{2}+\cdots+g_{2}^{n-1}=u_{n}(\beta)
  \begin{pmatrix}
- u_{n-2}(\beta) & u_{n-1}(\beta)\\
-\beta u_{n-1}(\gamma) & u_{n}(\beta)
\end{pmatrix}
  \]
\end{itemize}
  \item On pose $g_{3}:=s_{2}s_{1}$. Dans la base $(c_{1},c_{2})$ de $K^{2}$ on a:
  \[
  \forall n\in N,g_{3}^{n}=
  \begin{pmatrix}
-u_{2n-1}(\alpha) & u_{2n}(\alpha)\\
-\alpha u_{2n}(\alpha) & u_{2n+1}(\alpha)
\end{pmatrix}
  \]
  \begin{itemize}
  \item Si $n$ est pair on a:
  \[
  Id+g_{3}+g_{3}^{2}+\cdots+g_{3}^{n-1}=u_{n}(\alpha)
  \begin{pmatrix}
-\alpha u_{n-2}(\alpha) & u_{n-1}(\alpha)\\
-\alpha u_{n-1}(\alpha) & \alpha u_{n}(\alpha)
\end{pmatrix}
  \]
  \item 
  Si $n$ est impair on a:
  \[
  Id+g_{3}+g_{3}^{2}+\cdots+g_{3}^{n-1}=u_{n}(\alpha)
  \begin{pmatrix}
- u_{n-2}(\alpha) & u_{n-1}(\alpha)\\
-\alpha u_{n-1}(\alpha) & u_{n}(\alpha)
\end{pmatrix}
  \]
\end{itemize}
\end{enumerate}
\end{proposition}
\begin{proof}
Les formules pour $g_{i}^{n}$ se montrent sans difficultés par récurrence sur $n$. Les autres formules résultent alors immédiatement des formules de  \cite{Z2}.
\end{proof}
Soit $\pi:G\to G/N=G'$ la projection canonique et soit $\tau:G'\to G$ une section $(\pi\circ\tau=id_{G'}$) choisie de telle sorte que l'on ait:
\[
\forall i (1\leqslant i \leqslant3),\tau(s_{i})=s_{i}':=s_{i}(\lambda_{i}c_{i})\: \text{avec}\: \lambda_{i}c_{i}\in N\: \text{et}\:\lambda_{i}\in K^{*},\tau(1)=1.
\]
\begin{proposition}
Avec les hypothèses et notations ci-dessus, on a $\forall n\in \mathbb{N}$:
\begin{enumerate}
  \item $(s_{1}'s_{2}')^{n}=(s_{1}s_{2})^{n}b_{n}$ avec
  \begin{itemize}
  \item si $n$ est pair:
  \[
  b_{n}=u_{n}(\alpha)((\alpha u_{n}(\alpha)\lambda_{1}+u_{n-1}(\alpha)\lambda_{2})c_{1}+\alpha(u_{n+1}(\alpha)\lambda_{1}+u_{n}(\alpha)\lambda_{2})c_{2});
  \]
  \item si $n$ est impair:
\[
b_{n}=u_{n}(\alpha)((u_{n}(\alpha)\lambda_{1}+u_{n-1}(\alpha)\lambda_{2})c_{1}+(\alpha u_{n+1}(\alpha)\lambda_{1}+u_{n}(\alpha)\lambda_{2})c_{2}).
\] 
\end{itemize}
  \item $(s_{1}'s_{3}')^{n}=(s_{1}s_{3})^{n}b_{n}'$ avec
  \begin{itemize}
  \item si $n$ est pair:
  \[
  b_{n}'=u_{n}(\beta)((\beta u_{n}(\beta)\lambda_{1}+u_{n-1}(\beta)\lambda_{3})c_{1}+\beta(u_{n+1}(\beta)\lambda_{1}+u_{n}(\beta)\lambda_{3})c_{3});
  \]
  \item si $n$ est impair:
\[
b_{n}'=u_{n}(\beta)((u_{n}(\beta)\lambda_{1}+u_{n-1}(\beta)\lambda_{3})c_{1}+(\beta u_{n+1}(\beta)\lambda_{1}+u_{n}(\beta)\lambda_{3})c_{3}).
\] 
\end{itemize}
 \item $(s_{2}'s_{3}')^{n}=(s_{2}s_{2})^{n}b_{n}''$ avec
  \begin{itemize}
  \item si $n$ est pair:
  \[
  b_{n}''=u_{n}(\gamma)((\gamma u_{n}(\gamma)\lambda_{2}+mu_{n-1}(\gamma)\lambda_{3}c_{2}+(lu_{n+1}(\gamma)\lambda_{2}+\gamma u_{n}(\gamma)\lambda_{3})c_{3});
  \]
  \item si $n$ est impair:
\[
b_{n}''=u_{n}(\gamma)((u_{n}(\gamma)\lambda_{2}+mu_{n-1}(\gamma)\lambda_{3})c_{2}+(l u_{n+1}(\gamma)\lambda_{2}+u_{n}(\gamma)\lambda_{3})c_{3}).
\] 
\end{itemize}
\end{enumerate}
\end{proposition}
\begin{proof}
Dans toute la suite on emploie la notation additive pour $N$. On a $s_{1}'s_{2}'=s_{1}\lambda_{1}c_{1}s_{2}\lambda_{2}c_{2}=s_{1}s_{2}(\lambda c_{1}+(\alpha\lambda_{1}+\lambda_{2})c_{2})$. Posons, pour cette démonstration uniquement, $a:=\lambda c_{1}+(\alpha\lambda_{1}+\lambda_{2})c_{2}$. On a $s_{1}'s_{2}'=s_{1}s_{2}a$, puis $(s_{1}'s_{2}')^{2}=s_{1}s_{2}as_{1}s_{2}a=(s_{1}s_{2})^{2}(g_{3}.a+a)$. On montre facilement par récurrence sur $n$ que $(s_{1}'s_{2}')^{n}=(s_{1}s_{2})^{n}(g_{3}^{n-1}.a+g_{3}^{n-2}.a+\cdots+g_{3}.a+a)$. Il en résulte que $(s_{1}'s_{2}')^{n}=(s_{1}s_{2})^{n}((Id+g_{3}+\cdots+g_{3}^{n-1}).a)$. Nous appliquons maintenant les résultats de la proposition 2. Les autres résultats se montrent de la même manière.
\end{proof}
\begin{corollary}
Avec les notations précédentes, si $s_{1}s_{2}$ est d'ordre $p$ (resp. si $s_{1}s_{3}$ est d'ordre $q$, resp. si $s_{2}s_{3}$ est d'ordre $r$) alors $s_{1}'s_{2}'$ est d'ordre $p$ (resp.  $s_{1}'s_{3}'$ est d'ordre $q$, resp.  $s_{2}'s_{3}'$ est d'ordre $r$).
\end{corollary}
\begin{proof}
Si $s_{1}s_{2}$ est d'ordre $p$ alors $\alpha$ est racine de $v_{p}(X)$ donc $u_{p}(\alpha)=0$ et $b_{p}=0$: $(s_{1}'s_{2}')^{p}=(s_{1}s_{2})^{p}=1$: $s_{1}'s_{2}'$ est d'ordre un diviseur de $p$. Si $s_{1}'s_{2}'$ est d'ordre $n$, alors les formules donnant $b_{n}$ montrent que $u_{n}(\alpha)=0$, donc $p$ divise $n$ et $p=n$: $s_{1}'s_{2}'$ est d'ordre $p$.
\end{proof}
En complément de la proposition 3, nous avons:
\begin{corollary}
Avec les notations précédentes:
\begin{enumerate}
  \item Si $p$ est pair, $p=2p_{1}$, alors
  \[
  b_{p_{1}}=\frac{2}{4-\alpha}((2\lambda_{1}+\lambda_{2})c_{1}+(\alpha\lambda_{1}+2\lambda_{2})c_{2})=(-2\lambda_{2},\frac{2(2\beta+\alpha l)}{4-\alpha}\lambda_{1}+\frac{2(\beta+2l)}{4-\alpha}\lambda_{2});
 \]
  \item si $q$ est pair, $q=2q_{1}$, alors
   \[
  b_{q_{1}}'=\frac{2}{4-\beta}((2\lambda_{1}+\lambda_{3})c_{1}+(\beta\lambda_{1}+2\lambda_{3})c_{3})=(\frac{2(2\alpha+\beta m}{4-\beta}\lambda_{1}+\frac{2(\alpha+2m}{4-\beta}\lambda_{3},-2\lambda_{3});
 \]
  \item si $r$ est pair, $r=2r_{1}$, alors
  \[
  b_{r_{1}}''=\frac{2}{4-\gamma}((2\lambda_{2}+m\lambda_{3})c_{1}+(l\lambda_{2}+2\lambda_{3})c_{3})=(-2\lambda_{2},-2\lambda_{3}).
  \]
\end{enumerate}
\end{corollary}
\begin{proof}
On ne fait la démonstration que du 1), les deux autres démonstrations se faisant de la même manière. Pour cela nous utilisons les formules de \cite{Z2}  proposition 10.\\
Distinguons deux cas suivant la parité de $p_{1}$.
\begin{itemize}
  \item Si $p$ est pair, d'après la proposition 3, on a:
  \[
  b_{p_{1}}=u_{p_{1}}(\alpha)((\alpha u_{p_{1}}(\alpha)\lambda_{1}+u_{p_{1}-1}(\alpha)\lambda_{2})c_{1}+\alpha(u_{p_{1}+1}(\alpha)\lambda_{1}+u_{p_{1}}(\alpha)\lambda_{2})c_{2}
  \]
  mais $\alpha(4-\alpha)u_{p_{1}}^{2}(\alpha)=4$, $(4-\alpha)u_{p_{1}}(\alpha)u_{p_{1}-1}(\alpha)=2$ et $u_{p_{1}+1}(\alpha)=u_{p_{1}-1}(\alpha)$ car $\alpha$ est racine de $v_{2p_{1}}(X)$. Nous en déduisons:
  \[
  b_{p_{1}}=(\frac{4}{4-\alpha}\lambda_{1}+\frac{2}{4-\alpha}\lambda_{2})c_{1}+(\frac{2\alpha}{4-\alpha}\lambda_{1}+\frac{4}{4-\alpha}\lambda_{2})c_{2}
  \]
  ou encore:
  \[
  b_{p_{1}}=\frac{2}{4-\alpha}((2\lambda_{1}+\lambda_{2})c_{1}+(\alpha\lambda_{1}+2\lambda_{2})c_{2}).
  \]
\item Si $p$ est impair, d'après la proposition 3, on a:
 \[
  b_{p_{1}}=u_{p_{1}}(\alpha)( u_{p_{1}}(\alpha)\lambda_{1}+u_{p_{1}-1}(\alpha)\lambda_{2})c_{1}+(\alpha u_{p_{1}+1}(\alpha)\lambda_{1}+u_{p_{1}}(\alpha)\lambda_{2})c_{2}
  \]
  mais $(4-\alpha)u_{p_{1}}^{2}(\alpha)=4$, $(4-\alpha)u_{p_{1}}(\alpha)u_{p_{1}-1}(\alpha)=2$, car $\alpha$ est racine de $v_{2p_{1}}(X)$. Nous en déduisons:
   \[
  b_{p_{1}}=(\frac{4}{4-\alpha}\lambda_{1}+\frac{2}{4-\alpha}\lambda_{2})c_{1}+(\frac{2\alpha}{4-\alpha}\lambda_{1}+\frac{4}{4-\alpha}\lambda_{2})c_{2}
  \]
  ou encore:
  \[
  b_{p_{1}}=\frac{2}{4-\alpha}((2\lambda_{1}+\lambda_{2})c_{1}+(\alpha\lambda_{1}+2\lambda_{2})c_{2}).
  \]
\end{itemize}
\end{proof}
Nous donnons maintenant le résultat principal de ce paragraphe.
\begin{theorem}
Si $2|pqr$, une condition nécéssaire pour que la suite exacte $(\star\star)$ soit scindée est que la relation $(\mathcal{E})$ entre $\lambda_{1}$, $\lambda_{2}$ et $\lambda_{3}$ soit satisfaite:
\[
(\mathcal{E}) \quad (4-\gamma)\lambda_{1}+(l+2)\lambda_{2}+(m+2)\lambda_{3}=-1.
\]
Plus précisément, on a l'équivalence entre $(\mathcal{E})$ et l'une des conditions:
\begin{enumerate}
  \item si $p$ est pair, $p=2p_{1}$: $y_{p_{1}}'^{2}=1$;
  \item si $q$ est pair, $q=2q_{1}$: $x_{p_{1}}'^{2}=1$;
  \item si $r$ est pair, $r=2r_{1}$: $t_{r_{1}}'^{2}=1$
\end{enumerate}
(On rappelle que $\lambda_{i}c_{i}\in[N,c_{i}], (1\leqslant i\leqslant3)$). (On a posé: $\forall k\in \mathbb{N}$ $t_{k}=s_{1}(s_{2}s_{3})^{k}$, $x_{k}=s_{2}(s_{3}s_{1})^{k}$ et $y_{k}=s_{3}(s_{1}s_{2})^{k}$.)
\end{theorem}
Pour la démonstration, nous avons besoin du résultat suivant:
\begin{proposition}
On garde les hypothèses et notations du théorème, si $p$, $q$ et $r$ sont pairs. Alors:
\begin{itemize}
  \item $x_{q_{1}}^{2}=\frac{2}{l+2}c_{2}$;
  \item $y_{p_{1}}^{2}=\frac{2}{m+2}c_{3}$;
  \item $t_{r_{1}}^{2}=\frac{-2}{4-\gamma}c_{1}$.
\end{itemize}
\end{proposition}
\begin{proof}
Nous ne montrons que la première équation, les autres se montrant de la même manière. Nous avons vu dans \cite{Z2}   proposition 15 que, dans la base $\mathcal{A}$ de $M$ on a:
\[
(s_{3}s_{1})^{q_{1}}=(s_{1}s_{3})^{q_{1}}=
\begin{pmatrix}
-1 & \frac{2(2\alpha+\beta m)}{4-\beta} & 0\\
0 & 1 & 0\\
0 & \frac{2(\alpha+2m)}{4-\beta} & -1
\end{pmatrix}
\]
\[
x_{q_{1}}=(x)=s_{2}(s_{3}s_{1})^{q_{1}}=
\begin{pmatrix}
1 & 0 & 0\\
1 & -1 & l\\
0 & 0 & 1
\end{pmatrix}
\begin{pmatrix}
-1 & \frac{2(2\alpha+\beta m)}{4-\beta} & 0\\
0 & 1 & 0\\
0 & \frac{2(\alpha+2m)}{4-\beta} & -1
\end{pmatrix}
=
\begin{pmatrix}
-1 &  \frac{2(2\alpha+\beta m)}{4-\beta} & 0\\
-1 & 3 & -l\\
0 & \frac{2(\alpha+2m)}{4-\beta} & -1
\end{pmatrix}
\]
\[
x_{q_{1}}^{2}=(s_{2}(s_{3}s_{1})^{q_{1}})^{2}=
\begin{pmatrix}
1- \frac{2(2\alpha+\beta m)}{4-\beta} &  \frac{4(2\alpha+\beta m)}{4-\beta} &  \frac{-2l(2\alpha+\beta m)}{4-\beta}\\
-2 & 5 & -2l\\
\frac{-2(\alpha+2m)}{4-\beta} & \frac{4(\alpha+2m)}{4-\beta} & 1-\frac{2l(\alpha+2m)}{4-\beta}
\end{pmatrix}
\]
Nous exprimons maintenant $x_{q_{1}}$ dans la base $\mathcal{B}$ et nous n'avons besoin que du coefficient de $b$ dans $x(a_{2})$ et $x(a_{3})$ sachant que 
\[
a_{1}=(\frac{1}{4-\gamma})b-(\frac{l+2}{4-\gamma})a_{2}-(\frac{m+2}{4-\gamma})a_{3}.
\]
Nous trouvons, en utilisant les relations $(\mathcal{T}), $ $x_{q_{1}}^{2}=\frac{2}{l+2}c_{2}$.
\end{proof}
\begin{proof}(du théorème)
Nous montrons que $(\mathcal{E})$ est équivalente à $t_{r_{1}}^{2}=\frac{-2}{4-\gamma}c_{3}$, les autres équivalences se montrant de la même manière. On suppose donc que $r=2r_{1}$. On a $t=(t_{r_{1}})=s_{1}(s_{2}s_{3})^{r_{1}}$ et $t'=(t_{r_{1}}')=s'_{1}(s'_{2}s'_{3})^{r_{1}}=s_{1}(\lambda_{1}c_{1})(s_{2}s_{3})^{r_{1}}b_{r_{1}}''$ avec $b_{r_{1}}''=(-2\lambda_{2},-2\lambda_{3})$. Comme $(s_{2}s_{3})^{r_{1}}$ opère comme $-1$ sur $N$, on a $\lambda_{1}c_{1}(s_{2}s_{3})^{r_{1}}=(s_{2}s_{3})^{r_{1}}(-\lambda_{1}c_{1})$ donc $t'=t(-\lambda_{1}c_{1}+b_{r_{1}}'')$.\\
On obtient:
\[
t'^{2}=t(-\lambda_{1}c_{1}+b_{r_{1}}'')s_{1}(s_{2}s_{3})^{r_{1}}(-\lambda_{1}c_{1}+b_{r_{1}}'').
\]
On a $s_{1}.b_{r_{1}}''=b_{r_{1}}''+\omega(b_{r_{1}}'')c_{1}$ avec $\omega(b_{r_{1}}'')=2(\frac{\lambda_{2}(l+2)+\lambda_{3}(m+2)}{4-\gamma})$ d'où

\begin{eqnarray*}
t'^{2} & = & ts_{1}(\lambda_{1}c_{1}+b_{r_{1}}''+\omega(b_{r_{1}}'')c_{1})(s_{2}s_{3})^{r_{1}}(-\lambda_{1}c_{1}+b_{r_{1}}'')\\
 & = & t^{2}(-\lambda_{1}c_{1}-b_{r_{1}}''-\omega(b_{r_{1}}'')c_{1}+\lambda_{1}c_{1}+b_{r_{1}}'')\\
 & = & t^{2}(-2\lambda_{1}-\omega(b_{r_{1}}'')c_{1})\\
 & = & (\frac{-2}{4-\gamma}-2\lambda_{1}-2(\frac{\lambda_{2}(l+2)+\lambda_{3}(m+2)}{4-\gamma}))c_{1}\\
 & = & -(\frac{2}{4-\gamma})[1+\lambda_{1}(4-\gamma)+\lambda_{2}(l+2)+\lambda_{3}(m+2)]c_{1}.
\end{eqnarray*}
On a alors $t_{r_{1}}'^{2}=1\Leftrightarrow 1+\lambda_{1}(4-\gamma)+\lambda_{2}(l+2)+\lambda_{3}(m+2)=0$.\\ C'est la relation $(\mathcal{E})$.
\end{proof}
Pour utiliser ce théorème, nous aurons besoin des résultats suivants pour lesquels on suppose que $N(G)\neq \{1\}$.
\begin{definition}
On définit $I_{j}:=\{\lambda | \lambda\in K, \lambda c_{j}\in[N,s_{j}]\}$ $1\leqslant j\leqslant 3$.
\end{definition}
On a alors la proposition:
\begin{proposition}
\begin{enumerate}
  \item $I_{1}$, $I_{2}$ et $I_{3}$ sont des $\mathcal{O}'(K)$-modules.
  \item On a les inclusions suivantes:
  \begin{itemize}
  \item $\alpha I_{1}\subset I_{2}\subset I_{1}$; $\beta I_{1}\subset I_{3}\subset I_{1}$.
  \item $lI_{2}\subset I_{3}; mI_{3}\subset I_{2}$.
\end{itemize}
  \item Si $\alpha$ (resp. $\beta$) est inversible, on a $I_{1}=I_{2}$ (resp. $I_{1}=I_{3}$); si $\gamma$ est inversible, on a $I_{2}=I_{3}$. Si deux parmi $\alpha$, $\beta$, $\gamma$  sont inversibles, on a $I_{1}=I_{2}=I_{3}=I$ et $I$ est un $\mathcal{O}(K)$-module.
\end{enumerate}
\end{proposition}
\begin{proof}
Que $I_{j}$ $(1\leqslant j \leqslant 3)$ soit un $\mathcal{O}'(K)$-module est clair d'après le théorème 1.\\
Soit $\lambda\in I_{2}$. Alors $(s_{1}-id).\lambda c_{2}=\lambda c_{1}$, donc comme $\lambda c_{2}\in N(G)$, nous voyons que $\lambda c_{1}\in [N(G),s_{1}]$ et $\lambda\in I_{1}$: $I_{2}\subset I_{1}$. Les autres inclusions se montrent de la même manière.\\
Si $\alpha$ est inversible, il existe un polynôme $P(X)\in \mathbb{Z}[X]$ tel que $P(\alpha)\alpha=1$. Alors $I_{1}\subset\alpha^{-1}I_{2}\subset \alpha^{-1}I_{1}$, d'où $I_{1}\subset P(\alpha)I_{2}\subset P(\alpha)I_{1}\subset I_{1}$ et $I_{1}=P(\alpha) I_{2}=P(\alpha)I_{1}$, donc $\alpha I_{1}=\alpha P(\alpha)I_{1}=I_{1}$. Comme $\alpha I_{1}\subset I_{2}\subset I_{1}$, nous voyons que $I_{1}=I_{2}$. Les autres égalités se montrent de la même manière.
\end{proof}
\section{Propriétés de certains éléments de $\mathcal{Z}'$}.
Lorsque $n=3$, on voit tout de suite que si $s$ est une réflexion de $GL(M)$ et $z$ un élément de  $\mathcal{Z}'$ alors si $s$ et $z$ commutent $sz$ est encore une réflexion de $GL(M)$ (ceci n'est plus vrai si $n>3$) (pour la définition de $\mathcal{Z}'$ voir \cite{Z4} formules (5), (6) et (7)).

Pour éviter d'avoir à refaire souvent certains calculs, nous les donnons ici et le lecteur pourra sauter ce paragraphe et y revenir en cas de besoin.

On appelle $z_{i}$ l'unique élément de $\mathcal{Z}'$ qui centralise $s_{j}$ et $s_{k}$ si $|\{i,j,k\}|=3$ et l'on pose $s_{ij}:=s_{i}z_{j}$ $(1\leqslant i\neq j\leqslant3)$.\\
On a les résultats suivants qui ne présentent pas de difficultés à vérifier.
\begin{align}
z_{1}(a_{1}) & = & -a_{1}+\frac{2b}{4-\gamma} & , & z_{1}(a_{2}) & = & -a_{2} & , & z_{1}(a_{3}) & = & -a_{3}\\
 z_{2}(a_{2}) & = & -a_{2}+\frac{2b}{l+2} & , & z_{2}(a_{1}) & = & -a_{1} & , & z_{2}(a_{3}) & = & -a_{3}\\
 z_{3}(a_{3}) & = & -a_{3}+\frac{2b}{m+2} & , & z_{3}(a_{1}) & = & -a_{1} & , & z_{3}(a_{2}) & = & -a_{2}
\end{align}
On a $z_{1}z_{2}=(\frac{2}{l+2},0)$, $z_{1}z_{3}=(0,\frac{2}{m+2})$ et $z_{2}z_{3}=(\frac{-2}{l+2},\frac{2}{m+2})$ et aussi $(s_{1}z_{1})^{2}=(\frac{-2}{4-\gamma})c_{1}$, $(s_{2}z_{2})^{2}=(\frac{2}{l+2})c_{2}$ et $(s_{3}z_{3})^{2}=(\frac{2}{m+2})c_{3}$ (voir la proposition 4).
\begin{align}
s_{21}(a_{1})  &  = & a_{1}+(\frac{m+2}{4-\gamma})(la_{2}+2a_{3}) & , &s_{21}(a_{2}) & = & a_{2} & , & s_{21}(a_{3}) & = & a_{3}-(la_{2}+2a_{3})\\
s_{31}(a_{1})  &  = & a_{1}+(\frac{l+2}{4-\gamma})(2a_{2}+ma_{3}) & , &s_{31}(a_{3}) & = & a_{3} & , & s_{31}(a_{2}) & = & a_{2}-(2a_{2}+ma_{3}).
\end{align}
\[
\begin{matrix}
  C(s_{1},s_{21})=4-\alpha, &  C(s_{1},s_{31})=4-\beta, & C(s_{21},s_{31})=\gamma.   
\end{matrix}
\]
\begin{align}
s_{12}(a_{1}) & = & a_{1} & , & s_{12}(a_{3}) & = & a_{3}-(\beta a_{1}+2a_{3}) & , & s_{12}(a_{2}) & = & a_{2}+(\frac{m+2}{l+2})(\beta a_{1}+2a_{3}) \\
s_{32}(a_{2}) & = &  a_{2} & , & s_{32}(a_{1}) & = & a_{1}-(2a_{1}+a_{3}) & , & s_{32}(a_{2}) & = & a_{2}+(\frac{4-\gamma}{l+2})(2a_{1}+a_{3}).
\end{align}
\[
\begin{matrix}
C(s_{2},s_{12})=4-\alpha, & C(s_{2},s_{32})=4-\gamma, & C(s_{12},s_{32})=\beta.
\end{matrix}
\]
\begin{align}
s_{13}(a_{1})& = & a_{1}& ,  &  s_{13}(a_{2}) & = & a_{2}-(\alpha a_{1}+2a_{2})& , &s_{13}(a_{3}) & = & a_{3}+(\frac{l+2}{m+2})(\alpha a_{1}+2a_{2})
  \end{align}
  \begin{align}
s_{23}(a_{2}) & = & a_{2} & , & s_{23}(a_{1}) & = & a_{1}-(2a_{1}+a_{2}) & , & s_{23}(a_{3})+(\frac{4-\gamma}{m+2})(2a_{1}+a_{2}).    
\end{align}
\[
\begin{matrix}
C(s_{3},s_{13})=4-\beta, & C(s_{3},s_{23})=4-\gamma, & C(s_{13},s_{23})=\alpha.
\end{matrix}
\]
\[
\begin{matrix}
C(s_{21},s_{12})=\alpha, & C(s_{21},s_{32})=\gamma & C(s_{21},s_{13})=\alpha, & C(s_{21},s_{23}=4;\\
C(s_{31},s_{12})=\beta, & C(s_{31},s_{32})=4, & C(s_{31},s_{13})=\beta, & C(s_{31},s_{23})=\gamma;\\
C(s_{12},s_{32})=\beta, & C(s_{12},s_{13})=4, & C(s_{12},s_{23})=\alpha, & C(s_{32},s_{13})=\beta\\
C(s_{32},s_{23})=\gamma, & C(s_{13},s_{23})=\alpha. &&
\end{matrix}
\]
\begin{proposition}
Soient $W(p,q,r)$ un groupe de Coxeter de rang $3$ et $R:W\to GL(M)$ une représentation de réflexion réductible. Soit $z_{i}$ $(1\leqslant i \leqslant 3)$ l'élément de $\mathcal{Z}'$ qui centralise $s_{j}$ et $s_{k}$ si $|\{i,j,k\}|=3$ et qui fixe $b$. On pose $G:= Im R$ et $\Gamma_{i}:=<g,z_{i}>$. Alors:
\begin{enumerate}
  \item $\Gamma_{1}=\Gamma_{2}=\Gamma_{3}$.
  \item Si $N(G)\neq{1}$, on a:
  \begin{enumerate}
  \item si $G\cap\mathcal{Z}'\neq\emptyset$, $G=\Gamma$;
  \item si $G\cap\mathcal{Z}'=\emptyset$, $G$ est d'indice $2$ dans $\Gamma$.
\end{enumerate}
  \item Si $N(G)=\{1\}$, on a $\Gamma=N(\Gamma)G<z>$ et l'extension
  \[
  \{1\}\to N(\Gamma)\to \Gamma\to\Gamma/N(\Gamma) \to \{1\}
  \]
  est non scindée.
\end{enumerate}
\end{proposition}
\begin{proof}
C'est une conséquence immédiate de la proposition 4, du corollaire 2 de \cite{Z4} et de ce qui précède.
\end{proof}
\section{Un groupe de réflexion affine est de manière naturelle quotient de plusieurs groupes de Coxeter de rang $3$ lorsque $2|pqr$.}
Pour toute la suite les résultats suivants seront utiles. On garde les hypothèses et notations du début du chapitre.\\
Pour indiquer la dépendance du polynôme $Q(X)$ par rapport à $\alpha$, $\beta$, $\gamma$ nous le noterons $Q(X):=Q(\alpha,\beta,\gamma)(X)$.
\begin{proposition}
On suppose que $\Delta=0$ et que les racines de $Q(\alpha,\beta,\gamma)(X)$ sont $\alpha l$ et $\beta m$. Alors:
\begin{enumerate}
  \item les racines de $Q(\alpha',\beta',\gamma)(X)$ sont $-l(\alpha+2m)$ et $-m(\beta+2l)$;
  \item les racines de $Q(\alpha',\beta,\gamma')(X)$ sont $-\beta(m+2)$ et $(-2\beta+\alpha l)$;
  \item les racines de $Q(\alpha,\beta',\gamma')(X)$ sont $-\alpha(l+2)$ et $(-2\alpha+\beta m)$.
\end{enumerate}
(On rappelle que $\alpha'=4-\alpha$, $\beta'=4-\beta$ et $\gamma'=4-\gamma$).
\end{proposition}
\begin{proof}
Un calcul facile en utilisant le fait que $\Delta=0$ et que $lm=\gamma$.
\end{proof}
\begin{proposition}
Soit $G:=<a,b|a^{2}=b^{2}=(ab)^{n}=1>$ un groupe diédral avec $n\geqslant 3$. Si $n$ est pair, on pose $z:=(ab)^{n_{2}}$; si $n$ est impair, soit $<z>$ un groupe d'ordre $2$ qui centralise $G$, dans ce cas, on pose $\Gamma:=G<z>$ C'est un groupe diédral d'ordre $4n$. Alors:
\begin{enumerate}
  \item Si $n\equiv 0 \pmod{4}$, on a $G=<az,bz>=<a,bz>=<az,b>$.
  \item Si $n\equiv 2 \pmod{2}$, on a $G=<az,bz>$ et les sous-groupes $<a,bz>$ et $<az,b>$ sont d'indice $2$ dans $G$. De plus $G=<az,b>\times<z>\\=<a,bz>,\times<z>$
  \item Si $n$ est impair, on a $\Gamma=<a,bz>=<az,b>$, le sous-groupe $<az,bz>$ est d'indice $2$ dans $\Gamma$ et $\Gamma=<az,bz> \times<z>$.
\end{enumerate}
\end{proposition}
\begin{proof}

\end{proof}
\begin{notation}
Pour la suite nous introduisons des notations: si $G$ est un groupe de réflexion affine de rang $3$, nous appelons $\mathcal{A}=(a_{1},a_{2},a_{3})$ une base adaptée, $\mathcal{P}(G)=\mathcal{P}(\alpha,\beta,\gamma;\alpha l(G),\beta m(G))$ son système de paramètres et $G=<s_{1},s_{2},s_{3}>$.
\begin{enumerate}
  \item Soit $z_{1}$ l'unique élément de $\mathcal{Z}'$ qui centralise $s_{2}$ et $s_{3}$. Nous considérons les trois groupes suivants:
  \begin{enumerate}
  \item $G_{1}':=<s_{1},s_{21},s_{31}>$, base adaptée: $\mathcal{A}(G_{1}')=(a_{1},(\frac{m+2}{4-\gamma})(la_{2}+2a_{3}),(\frac{l+2}{4-\gamma})(2a_{2}+ma_{3}))$ et le système de paramètres:\\ $\mathcal{P}(G_{1}')=\mathcal{P}(\alpha',\beta',\gamma;-m(\beta+2l),-l(\alpha+2m))$.
  \item $\Gamma_{1}:=<s_{1},s_{2},s_{31}>$, base adaptée: $\mathcal{A}(\Gamma_{1})=(a_{1},a_{2},(\frac{l+2}{4-\gamma})(2a_{2}+ma_{3}))$ et le sytème de paramètres: $\mathcal{P}(\Gamma_{1})=\mathcal{P}(\alpha,\beta',\gamma';-\alpha(l+2),-(2\alpha+\beta m))$.
  \item $\Gamma_{1}':=<s_{1},s_{21},s_{3}>$, base adaptée: $\mathcal{A}(\Gamma_{1}')=(a_{1},(\frac{m+2}{4-\gamma})(la_{2}+2a_{3}),a_{3})$ et le sytème de paramètres: $\mathcal{P}(\Gamma_{1}')=\mathcal{P}(\alpha',\beta,\gamma';-(2\beta+\alpha l),-\beta(m+2))$.
\end{enumerate}
  \item Soit $z_{2}$ l'unique élément de $\mathcal{Z}'$ qui centralise $s_{1}$ et $s_{3}$. Nous considérons les trois groupes suivants:
  \begin{enumerate}
  \item $G_{2}':=<s_{12},s_{2},s_{32}>$, base adaptée: $\mathcal{A}(G_{2}')=(\beta a_{1}+2a_{3},(\beta+2l)a_{2},-\beta(2a_{1}+a_{3}))$ et le système de paramètres:\\ $\mathcal{P}(G_{2}')=\mathcal{P}(\alpha',\beta,\gamma';-\beta (m+2),-(\alpha l+2\beta))$.
  \item $\Gamma_{2}:=<s_{1},s_{2},s_{32}>$, base adaptée: $\mathcal{A}(\Gamma_{2})=(a_{1},a_{2},-(2a_{1}+a_{3}))$ et le sytème de paramètres: $\mathcal{P}(\Gamma_{2})=\mathcal{P}(\alpha,\beta',\gamma';-\alpha(l+2),-(2\alpha+\beta m))$.
  \item $\Gamma_{2}':=<s_{12},s_{2},s_{3}>$, base adaptée:\\ $\mathcal{A}(\Gamma_{2}')=(\beta a_{1}+2a_{3},(\beta+2l)a_{2},-(4-\beta)a_{3})$ et le sytème de paramètres: $\mathcal{P}(\Gamma_{2}')=\mathcal{P}(\alpha',\beta',\gamma;-l(\alpha +2m),-m(\beta+2l))$.
\end{enumerate}
  \item Soit $z_{3}$ l'unique élément de $\mathcal{Z}'$ qui centralise $s_{1}$ et $s_{2}$. Nous considérons les trois groupes suivants:
  \begin{enumerate}
  \item $G_{3}':=<s_{13},s_{23},s_{3}>$, base adaptée:\\ $\mathcal{A}(G_{3}')=(\alpha a_{1}+2a_{2},-\alpha(2a_{1}+a_{2}),(\alpha+2m)a_{3})$ et le système de paramètres: $\mathcal{P}(G_{3}')=\mathcal{P}(\alpha,\beta',\gamma';-(2\alpha+\beta m),-\alpha (l+2))$.
  \item $\Gamma_{3}:=<s_{1},s_{23},s_{3}>$, base adaptée: $\mathcal{A}(\Gamma_{3})=(a_{1},-(2a_{1}+a_{2}),a_{3})$ et le système de paramètres:\\ $\mathcal{P}(\Gamma_{3})= \mathcal{P}(\alpha',\beta,\gamma';-(2\beta+\alpha l),-\beta(m+2))$. 
  \item $\Gamma_{3}':=<s_{13},s_{2},s_{3}>$, base adaptée:\\ $\mathcal{A}(\Gamma_{3}')=(\alpha a_{1}+2a_{2},-(4-\alpha)a_{2}),(\alpha+2m)a_{3})$ et le système de paramètres: $\mathcal{P}(\Gamma_{3}')= \mathcal{P}(\alpha',\beta',\gamma;-l(2m+\alpha),-m(\beta+2l))$
\end{enumerate}
\end{enumerate}
\end{notation}
\begin{proposition}
Nous gardons les notations précédentes. Le groupe $G$ opère comme un groupe de réflexion sur l'espace $M^{*}$ dual de l'espace $M$. Pour cette opération de $G$, le système de paramètres est $\mathcal{P}^{*}(G)=\mathcal{P}^{*}(\alpha,\beta,\gamma;\beta m,\alpha l)$.
\end{proposition}
\begin{proof}
On sait que dans la base duale $\mathcal{A}^{*}$ de la base $\mathcal{A}$, les matrices des éléments de $G$ sont les transposées des matrices dans la base $\mathcal{A}$. $G$ opère donc comme un groupe de réflexion sur $M^{*}$ et en appelant $(e_{1},e_{2},e_{3})$ la base base duale de $\mathcal{A}$ $(e_{i}(a_{j})=\delta_{ij})$, on a:
\[
\begin{array}{c|ccc}
& s_{1}.e_{1} & = & -e_{1}\\
& s_{1}.e_{2} & = & e_{2}+ e_{1}\\
& s_{1}.e_{3} & = & e_{3}+ e_{1}
\end{array}
,
\begin{array}{c|ccc}
& s_{2}.e_{1} & = & e_{1}+ \alpha e_{2}\\
& s_{2}.e_{2} & = & -e_{2}\\
& s_{2}.e_{3} & = & e_{3}+ me_{2}
\end{array}
,
\begin{array}{c|ccc}
& s_{3}.e_{1} & = & e_{1}+ \beta e_{3}\\
& s_{3}.e_{2} & = & e_{2}+ le_{3}\\
& s_{3}.e_{3} & = & -e_{3}
\end{array}
.
\]
Posons $e_{2}':=s_{2}(e_{1})-e_{1}=\alpha e_{2}$ et $e_{3}':=s_{3}(e_{1})-e_{1}=\beta e_{2}$. Nous avons $s_{2}(e_{3}')=\beta me_{2}+\beta e_{3}=\frac{\beta m}{\alpha}e_{2}'+e_{3}'$ et $s_{3}(e_{3}')=\alpha e_{2}+\alpha le_{3}=e_{2}'+\frac{\alpha l}{\beta}e_{3}'$ donc $\mathcal{P}^{*}(G)=\mathcal{P}^{*}(\alpha,\beta,\gamma;\beta m,\alpha l)$.
\end{proof}
\begin{theorem}
Soit $W(p,q,r)$ un groupe de Coxeter de rang $3$. On suppose que $2|pqr$. Soit $R:W(p,q,r)\to GL(M)$ l'une de ses représentations de réflexion affine avec $G:=Im R$ et le système de paramètres $\mathcal{P}(G)=\mathcal{P}(\alpha,\beta,\gamma;\alpha l, \beta m)$. Alors le groupe $G$ est aussi groupe de réflexion affine quotient d'un des groupes de Coxeter suivant:
\begin{enumerate}
  \item $W(p',q',r)$ avec $\mathcal{P_{1}}(G)=\mathcal{P_{1}}(\alpha',\beta',\gamma;-m(\beta+2l),-l(\alpha+2m))$.
  \item $W(p',q,r')$ avec $\mathcal{P_{2}}(G)=\mathcal{P_{2}}(\alpha',\beta,\gamma';-\beta (m+2),-(\alpha l+2\beta))$.
  \item $W(p,q',r')$ avec $\mathcal{P}(G)=\mathcal{P}(\alpha,\beta',\gamma';-\alpha(l+2),-(2\alpha+\beta m))$.
\end{enumerate}
\end{theorem}
\begin{proof}
On peut supposer par exemple que$2|r$. Soit $z_{1}=(s_{2}s_{3})^{\frac{r}{2}}$. D'après la proposition 8, on a $<s_{2},s_{3}>=<s_{2}z_{1},s_{3}z_{1}>$ donc $G=<s_{1},s_{21},s_{31}>$   et $\mathcal{P_{1}}(G)=\mathcal{P_{1}}(\alpha',\beta',\gamma;-m(\beta+2l),-l(\alpha+2m))$ d'après le résultat 1. (a) plus haut. On fait de même avec $z_{2}$ et $z_{3}$ pour avoir le résultat.
\end{proof}
\section{Quelques présentations des groupes de réflexion affines de rang $3$ lorsque $2|pqr$. }
Nous montrons dans ce paragraphe que pour tout groupe de Coxeter de rang $3$, $W(p,q,r)$, avec $2|pqr$, il existe une représentation de réflexion affine ``universelle'' en ce sens que toute représentation de réflexion affine de $W(p,q,r)$ est quotient de celle-là.
\begin{proposition}
Soient $W(p,q,r)$ un groupe de Coxeter de rang $3$ et $R:W(p,q,r)\to GL(M)$ une représentation de réflexion de ce groupe. on pose: $G:=Im R$. On suppose que $p=2p_{1}$, $q=2q_{1}$ et $r=2r_{1}$. On a l'équivalence des conditions (A) et (B):
\begin{itemize}
  \item (A) \quad $\Delta(G)=0$ (i.e. $R$ est réductible);
  \item (B) \quad $(s_{1}s_{2})^{p_{1}}(s_{1}s_{3})^{q_{1}}(s_{2}s_{3})^{r_{1}}=(s_{2}s_{3})^{r_{1}}(s_{1}s_{2})^{p_{1}}(s_{1}s_{3})^{q_{1}}$.
 \end{itemize}
\end{proposition}
\begin{proof}
1) Si $\Delta(G)=0$, on a déjà vu que $(s_{1}s_{2})^{p_{1}}(s_{1}s_{3})^{q_{1}}$ et $(s_{1}s_{2})^{p_{1}}(s_{2}s_{3})^{r_{1}}$ sont dans $N(G)$. Comme celui-ci est commutatif on a:
\[
(s_{1}s_{2})^{p_{1}}(s_{1}s_{3})^{q_{1}}(s_{1}s_{2})^{p_{1}}(s_{2}s_{3})^{r_{1}}=(s_{1}s_{2})^{p_{1}}(s_{2}s_{3})^{r_{1}}(s_{1}s_{2})^{p_{1}})(s_{1}s_{3})^{q_{1}}
\]
donc, après simplification par $(s_{1}s_{2})^{p_{1}}$ on obtient le résultat.\\
2) On applique la construction fondamentale. On a la base adaptée $\mathcal{A}$ et le système de paramètres $\mathcal{P}(G)=\mathcal{P}(\alpha,\beta,\gamma;\alpha l,\beta m)$.\\
D'après la proposition B1 de l'appendice B, on a:
\[
(s_{1}s_{3})^{q_{1}}=
\begin{pmatrix}
-1 & \frac{2(2\alpha+\beta} {4-\beta} & 0\\
0 & 1 & 0\\
0 & \frac{2(\alpha+2m)}{4-\beta} & -1
\end{pmatrix}
,
(s_{1}s_{2})^{p_{1}}=
\begin{pmatrix}
-1 & 0 & \frac{2(2\beta+\alpha l)}{4-\alpha}\\
0 & -1 & \frac{2(\beta+2l)}{4-\alpha}\\
0 & 0 & 1
\end{pmatrix}
\]
\[
\text{et}\quad
(s_{2}s_{3})^{r_{1}}=
\begin{pmatrix}
-1 & 0 & 0\\
\frac{2(l+2)}{4-\gamma} & -1 & 0\\
\frac{2(m+2)}{4-\gamma} & 0 & -1
\end{pmatrix}
.
\]
Nous obtenons:
\[
(s_{1}s_{3})^{q_{1}}(s_{1}s_{2})^{p_{1}}(s_{2}s_{3})^{r_{1}}=
\]
\[
\begin{pmatrix}
1 & \frac{-2(2\alpha+\beta m)}{4-\beta} & \frac{-2(2\beta+\alpha l)}{4-/beta}+\frac{4(2\alpha+\beta m)(\beta+2l)}{(4-\alpha)(4-\beta)}\\
0 & -1 & \frac{2(\beta+2l)}{4-\alpha}\\
0 & \frac{-2(\alpha+2m)}{4-\beta} & -1+\frac{4(\alpha+2m)(\beta+2l)}{(4-\alpha)(4-\beta)}
\end{pmatrix}
\begin{pmatrix}
1 & 0 & 0\\
\frac{2(l+2)}{4-\gamma} & -1 & 0\\
\frac{2(m+2)}{4-\gamma} & 0  & -1
\end{pmatrix}
\]
et 
\[
(s_{2}s_{3})^{r_{1}}(s_{1}s_{2})^{p_{1}}(s_{1}s_{3})^{q_{1}}=
\]
\[
\begin{pmatrix}
-1 & 0 & \frac{2(2\beta+\alpha l)}{4-\alpha}\\
\frac{-2(l+2)}{4-\gamma} & 1 & \frac{-2(\beta+2l)}{4-\alpha}+\frac{4(l+2)(2\beta+\alpha l)}{(4-\alpha)(4-\gamma)}\\
\frac{-2(m+2)}{4-\gamma} & 0 & -1+\frac{4(m+2)(2\beta+\alpha l)}{(4-\alpha)(4-\gamma)}
\end{pmatrix}
\begin{pmatrix}
-1 & \frac{2(2\alpha+\beta m)}{4-\beta} & 0\\
0 & 1 & 0\\
0 & \frac{2(\alpha+2m)}{4-\beta} & -1
\end{pmatrix}
.
\]
On écrit que le coefficient de la deuxième ligne et de la troisième colonne de ces deux matrices sont égaux:
\[
\frac{-2(\beta+2l)}{4-\alpha}=\frac{2(\beta+2l)}{4-\alpha}-\frac{4(l+2)(2\beta+\alpha l)}{(4-\alpha)(4-\gamma)}
\]
d'où, après simplifications, on trouve l'égalité: $(4-\gamma)(\beta+2l)=(l+2)(2\beta+\alpha l)$. C'est l'une des relations $(\mathcal{T})$ et nous savons que cela implique que $\Delta(G)=0$.
\end{proof}
\begin{bfseries}
Tous les résultats dans ce qui suit sont obtenus après avoir appliqué la construction fondamentale.
\end{bfseries}
\begin{corollary}
Soit $W(p,q,r)$ un groupe de Coxeter de rang $3$. Soit $H$ le quotient de $W(p,q,r)$ qui a la présentation:
\[
(w(p,q,r),((s_{1}s_{3})^{q_{1}}(s_{1}s_{2})^{p_{1}}(s_{2}s_{3})^{r_{1}})^{2}=1).
\]
Pour toute représentation de réflexion affine $R$ de $W(p,q,r)$, si $G:=Im R$, alors $G$ est isomorphe à un quotient de $H$.
\end{corollary}
\begin{proof}
C'est clair car dans la proposition précédente, les paramètres $\alpha$, $\beta$, $\gamma$, $\alpha l$ et $\beta m$ n'interviennent pas.
\end{proof}
On rappelle que l'on a posé $t_{k}=s_{1}(s_{2}s_{3})^{k}$, $x_{k}=s_{2}(s_{3}s_{1})^{k}$ et $y_{k}=s_{3}(s_{1}s_{2})^{k}$. On peut remarquer que si $r=2r_{1}$, alors $t_{r_{1}}^{2}$ est dans le groupe $N(G)$ si $\Delta(G)=0$; même chose si $p$ ou $q$ est pair. Pour simplifier les notations, si $r=2r_{1}$ on pose $t':=t_{r_{1}}$. On utilise toujours le fait que $N(G)$ est commutatif dans les propositions suivantes.
\begin{proposition}
Soit $W(p,q,r)$ un groupe de Coxeter de rang $3$. On suppose que $r$ est pair: $r=2r_{1}$ et que $p$, $q$ et $r$ sont $\geqslant 3$.
\begin{enumerate}
  \item On appelle $G_{2}$ le quotient de $W(p,q,r)$ lorsque l'on ajoute la relation:
  \[
  t'^{2}(s_{2}t'^{2}s_{2})=(s_{2}t'^{2}s_{2})t'^{2}.
  \]
  Soit $R_{2}$ une représentation de réflexion de $G_{2}$ avec le système de paramètres $\mathcal{P}(G_{2})=\mathcal{P}(\alpha,\beta,\gamma;\alpha l,\beta m)$. Alors on a l'une des possibilités suivantes:
  \begin{enumerate}
  \item $R_{2}$ est réductible (i. e. $\Delta(G_{2})=0$);
  \item $\alpha=1$ et $\Delta(G_{2})=4-\gamma$ (on a alors $p=3$);
  \item $\alpha=2$, $\beta=1$, $\gamma=2$, $l=-1$ et $m=-2$ (on a alors $p=r=4$ et $q=3$) et $G_{2}\simeq W(B_{3})$.
\end{enumerate}
  \item On appelle $G_{3}$ le quotient de $W(p,q,r)$ lorsque l'on ajoute la relation:
  \[
  t'^{2}(s_{3}t'^{2}s_{3})=(s_{3}t'^{2}s_{3})t'^{2}.
  \]
  Soit $R_{3}$ une représentation de réflexion de $G_{3}$ avec le système de paramètres $\mathcal{P}(G_{3})=\mathcal{P}(\alpha,\beta,\gamma;\alpha l,\beta m)$. Alors on a l'une des possibilités suivantes:
   \begin{enumerate}
  \item $R_{3}$ est réductible (i. e. $\Delta(G_{3})=0$);
  \item $\beta=1$ et $\Delta(G_{3})=4-\gamma$ (on a alors $q=3$);
  \item $\alpha=1$, $\beta=2$, $\gamma=2$, $l=-2$ et $m=-1$ (on a alors $q=r=4$ et $p=3$) et $G_{3}\simeq W(B_{3})$.
\end{enumerate}
  \end{enumerate}
\end{proposition}
\begin{proof}
La démonstration du 2) est semblable à celle du 1). Dans toute la suite de cette démonstration pour simplifier les notations, on pose: $T:=\frac{\Delta}{4-\gamma}$ (où $\Delta:=\Delta(G_{2})$). On peut remarquer que $\alpha(l+2)+\beta(m+2)=2(4-\gamma)-\Delta$ et $\frac{2}{4-\gamma}(\alpha(l+2)+\beta(m+2))=4-2T$.\\
D'après la proposition B1 on a:
\[
t'=
\begin{pmatrix}
3-2T & -\alpha & -\beta\\
\frac{2(l+2)}{4-\gamma} & -1 & 0\\
\frac{2(m+2)}{4-\gamma} & 0 & -1
\end{pmatrix}
\]
donc 
\[
t'^{2}=
\begin{pmatrix}
4T^{2}-10T+5 & -2\alpha (1-T) & -2\beta(1-T)\\
\frac{4(l+2)}{4-\gamma}(1-T) & 1-\frac{2\alpha(l+2)}{4-\gamma} & -\frac{2\beta(l+2)}{4-\gamma}\\
\frac{4(m+2)}{4-\gamma}(1-T) & -\frac{2\alpha(m+2)}{4-\gamma} & 1-\frac{2\beta(m+2)}{4-\gamma}
\end{pmatrix}
.
\]
Si $b_{1}=(4-\gamma)a_{1}+(l+2)a_{2}+(m+2)a_{3}$, on a $<b_{1}>=C_{M}(s_{2},s_{3})$ et $s_{1}(b_{1})=b_{1}-\Delta a_{1}=b_{1}-(4-\gamma)Ta_{1}$.\\
Pour encore simplifier les notations nous posons $B:=\frac{2}{4-\gamma}b_{1}$. Nous obtenons alors:
\begin{eqnarray*}
t'^{2}(a_{1}) & = & (4T^{2}-6T+1)a_{1}+2(1-T)B; \\
t'^{2}(a_{2}) & = & 2\alpha Ta_{1}+a_{2}-\alpha B;\\
t'^{2}(a_{3}) & = & 2\beta Ta_{1}+a_{3}-\beta B;\\
t'^{2}(B) & = & 4T(T-1)a_{1}+(1-2T)B.
\end{eqnarray*}
Puis
\begin{eqnarray*}
s_{2}t'^{2}s_{2}(a_{1}) & = & (4T^{2}+2(\alpha-3)T+1)a_{1}+2T(2T+(\alpha-3))a_{2} +(2-\alpha -2T)B;\\
s_{2}t'^{2}s_{2}(a_{2}) & = & -2\alpha Ta_{1}+(1-2\alpha T)a_{2}+\alpha B;\\
s_{2}t'^{2}s_{2}(a_{3}) & = & 2(\beta+\alpha l)Ta_{1}+2(\beta+\alpha l)Ta_{2}+a_{3}-(\beta+\alpha l)B;\\
s_{2}t'^{2}s_{2}(B) & = & 4T(T-1)a_{1}+4T(T-1)a_{2}+(1-2T)B.
\end{eqnarray*}
On obtient:
\begin{eqnarray*}
t'^{2}(s_{2}t'^{2}s_{2}(a_{3})) & = & 2(\beta+\alpha l)T((4T^{2}-6T+1)a_{1}+2(1-T)B) \\
 & + & 2(\beta+\alpha l)T(2\alpha Ta_{1}+a_{2}-\alpha B)+(2\beta Ta_{1}+a_{3}-\beta B)\\
 & - & (\beta+\alpha l)(4T(T-1)a_{1}+(1-2T)B) 
\end{eqnarray*}
et 
\begin{eqnarray*}
(s_{2}t'^{2}s_{2})t'^{2}(a_{3}) & = & 2\beta T [(4T^{2}+2(\alpha-3)T+1)a_{1}+2T(2T+(\alpha-3))a_{2}+(2-\alpha-2T)B] \\
 & + & [2(\beta+\alpha l)Ta_{1}+2(\beta+\alpha l)Ta_{2}+a_{3}-(\beta+\alpha l)B\\
 & - & \beta[4T(T-1)a_{1}+4T(T-1)a_{2}+(1-2T)B].
\end{eqnarray*}
On en déduit, après simplifications:
\[
(t'^{2}(s_{2}t'^{2}s_{2}(a_{3})-(s_{2}t'^{2}s_{2})t'^{2})a_{3}=0=
\]
\[
4\alpha lT(2T^{2}-(4-\alpha)T+1)a_{1}-4\beta T(2T^{2}-(4-\alpha)T+1)a_{2}-2\alpha lT(2T+\alpha-3)B.
\]
On a donc, ou bien $T=0$, c'est à dire $R_{2}$ est réductible, ou bien
\[
(\star) \qquad (2T^{2}-(4-\alpha)T+1)(2\alpha la_{1}-2\beta a_{2})-2\alpha lT(2T+\alpha-3)B=0.
\]
Dans la suite on suppose que $T\neq 0$. On a $B=\frac{2}{4-\gamma}b_{1}=2a_{1}+\frac{2(l+2)}{4-\gamma}a_{2}+\frac{2(m+2)}{4-\gamma}a_{3}$ et la relation $(\star)$ devient, en donnant à B sa valeur:
\begin{eqnarray*}
2\alpha l(2T^{2}-(6-\alpha)T+(4-\alpha))a_{1} &  \\
-2 (\beta(2T^{2}-(4-\alpha)T+1)+\alpha l(2T+(\alpha-3))(\frac{l+2}{4-\gamma}))a_{2} & \\
-2\alpha l(2T+(\alpha-3))(\frac{m+2}{4-\gamma})a_{3}=0.
\end{eqnarray*}
Comme $(a_{1},a_{2},a_{3})$ est une base de $M$, nous obtenons les trois équations suivantes qui doivent être satisfaites:
\begin{enumerate}
  \item $2\alpha l(2T^{2}-(6-\alpha)T+(4-\alpha))=0$
  \item $-2[\beta(2T^{2}-(4-\alpha)T+1)+\alpha (\frac{l+2}{4-\gamma})(2T+(\alpha-3))]=0$
  \item $(\frac{-2\alpha l(m+2)}{4-\gamma})(2T+(\alpha-3))=0$
\end{enumerate}
Comme $2\alpha l\neq 0$puisque $lm=\gamma\neq 0$ 1. peut s'écrire $(T-1)(2T-(4-\alpha))=0$ d'où deux possibilités:\\
\textbf{Première possibilité}: $T=1$. Alors 2. devient $(\frac{\alpha-1}{4-\gamma})(\beta(4-\alpha)+\alpha l(l+2))=0$ et 3. devient $(\alpha-1)(m+2)=0$, donc ou bien $\alpha=1$, ou bien $\alpha \neq 1$, $m=-2$, $4-\gamma=2(l+2)$ car $\gamma=lm$ et 2. devient $(2\beta+\alpha l)(m+2)=0$. Comme $\gamma \neq 4$ on a $l\neq -2$ et $2\beta+\alpha l=0$ d'où $\alpha l= \beta m$ et $\alpha\gamma=4\beta$. D'après la proposition A17, nous avons $\alpha=\gamma=2$ et $\beta=1$ dans ce cas.\\
\textbf{Deuxième possibilité}: $T=\frac{4-\alpha}{2}$. Alors 3. devient $m+2=0$, $m=-2$,\\ $4-\gamma=2(l+2)$ ($l\neq-2$) et 2. devient : $-2\beta(\frac{2(4-\alpha)^{2}}{4}-\frac{(4-\alpha)^{2}}{2}+1)-\alpha l=0$ d'où $2\beta+\alpha l=0$. Comme ci-dessus $\alpha l=\beta m$, $\alpha\gamma=4\beta$, $\alpha=\gamma=2$, $\beta=1$ et $T=1$. Nous obtenons $p=r=4$ et $q=3$. De plus $\beta+l=\alpha+m=0$ donc $s_{2}s_{3}^{s_{1}}$ est d'ordre $2$, $G=<s_{2},s_{1},s_{3}^{s_{1}}>$ donc $G\simeq W(B_{3})$.

On vérifie sans peine que, dans tous les cas,
\[
t'^{2}(s_{2}t'^{2}s_{2})(a_{i})=(s_{2}t'^{2}s_{2})t'^{2}(a_{i}) \quad \text{pour}\quad  1\leqslant i \leqslant 3.
\]
En effet, il suffit de le faire lorsque $T=1$ avec les valeurs indiquées pour les paramètres $\alpha,\beta,\gamma, \alpha l, \beta m$.
\end{proof}
On a des résultats analogues lorsque $p$ est pair (resp. $q$ est pair).
\begin{corollary}
Soit $W(p,q,r)$ un groupe de Coxeter de rang $3$. On suppose que $r$ est pair, que $p\geqslant 4$ et si $p=4$ que $q\geqslant 4$. Soit $R$ un représentation de réflexion de $W(p,q,r)$. Alors:
\begin{enumerate}
  \item Les deux conditions suivantes sont équivalentes:
  \begin{itemize}
  \item (A) \qquad $\Delta=0$;
  \item (B) \qquad $(t'^{2}s_{2})^{2}=(s_{2}t'^{2})^{2}$.
\end{itemize}
  \item On suppose que les conditions (A) et (B) sont satisfaites. Alors il existe une représentation de réflexion affine $R$ de $W(p,q,r)$ telle que si $G=Im R$, pour toute représentation affine $R'$ de $W(p,q,r)$ avec $G'=Im R'$, alors $G'$ est isomorphe à un quotient de $G$. Si $G=<s_{1},s_{2},s_{3}>$, une ``présentation'' de $G$ est:
  \[
  (w(p,q,r),(t'^{2}s_{2})^{2}=(s_{2}t'^{2})^{2}).
  \]
\end{enumerate}
\end{corollary}
\begin{proof}
Il est clair que (A) $\Rightarrow$ (B) car $N$ est commutatif et $t'^{2}$ et $s_{2}t'^{2}s_{2}$ sont dans $N$. Supposons la condition (B) satisfaite. On a le résultat grâce à la proposition 11. Le 2. est clair car les paramètres n'interviennent pas.
\end{proof}
On a des résultats analogues en échangeant les rôles de $p$, $q$ et $r$.
\begin{proposition}
Soit$W(p,q,r)$ un groupe de Coxeter de rang $3$. On suppose que $r$ est pair, et $p\geqslant 4$. On appelle $G$ le quotient de $W(p,q,r)$ lorsque l'on ajoute les relations:
\[
(1) \quad (t'^{2}s_{2})^{2}=(s_{2}t'^{2})^{2} \quad \text{et} \quad (2) \quad (t'^{2}s_{3})^{2}=(s_{3}t'^{2})^{2}.
\]
Soit $R$ une représentation de réflexion de $G$. Alors on a l'une des possibilités suivantes:
\begin{enumerate}
  \item $R$ est réductible (i.e. $\Delta=0$).
  \item $G$ est fini et on a l'une des possibilités suivantes:\begin{enumerate}
  \item $p=q=3$, $\alpha=\beta=1$, $\Delta=4-\gamma$ et $G\simeq G(r,r,3)$ groupe de réflexion complexe imprimitif.
  \item $p=3$, $q=r=4$, $\alpha=1$, $\beta=\gamma=2$, $l=-2$, $\Delta=2$, $\alpha l= \beta m$ et $G\simeq W(B_{3})$.
  \item $p=r=4$, $q=3$, $\alpha=\gamma=2$, $\beta=1$, $l=-1$, $\Delta=2$, $\alpha l= \beta m$ et $G\simeq W(B_{3})$.
\end{enumerate}
 \end{enumerate}
\end{proposition}
\begin{proof}
D'après la proposition 11 la relation (1) implique que l'on a l'une des possibilités:\\ $\Delta=0$ ou $(\Delta=4-\gamma\; \text{et}\; \alpha=1)$ ou $(m=-2,\Delta=\frac{1}{2}(4-\alpha)(4-\gamma),\alpha l=\beta m, \;\text{on a alors}\; l\neq-2)$.\\
De même la relation (2) implique que l'on a l'une des possibilités:\\ $\Delta=0$ ou $(\Delta=4-\gamma\; \text{et}\; \beta=1)$ ou $(m=-2,\Delta=\frac{1}{2}(4-\beta)(4-\gamma),\alpha l=\beta m, \;\text{on a alors}\; m\neq-2)$.\\
On peut donc avoir $\Delta=0$: $R$ est réductible; ou bien $\Delta=4-\gamma$, $\alpha=\beta=1$ (i.e. $p=q=3$), alors d'après la proposition 5.7, $G\simeq G(r,r,3)$; ou bien $\Delta=4-\gamma$, $\alpha=1$, $l=-2$, $\Delta=4-\gamma=\frac{1}{2}(4-\beta)(4-\gamma)$ donc $\beta=2$, $\alpha l=\beta m=-2=2m$, $m=-1$ et $\gamma=lm=2$. On a $p=3$, $q=r=4$ et $G\simeq W(B_{3})$.\\
Le quatrième cas est identique au précédent.
\end{proof}
Notons enfin le résultat suivant:
\begin{proposition}
Soient $p$, $q$ et $r$ trois entiers impairs $\geqslant3$. Soit $R$ une représentation de réflexion de $W(2p,2q,r)$. Alors:
\begin{enumerate}
  \item Les trois conditions suivantes sont équivalentes:
  \begin{itemize}
  \item (A) \qquad $R$ est réductible (i.e. $\Delta=0$);
  \item (B) \qquad $(y'^{2}s_{1})^{2}=(s_{1}y'^{2})^{2}$;
  \item (C) \qquad $(x'^{2}s_{1})^{2}=(s_{1}x'^{2})^{2}$.
\end{itemize}
  \item On suppose les conditions (A), (B) et (C) satisfaites. Il existe une représentation de réflexion affine $R$ de $W(2p,2q,r)$ telle que si $G:=Im R$, pour toute représentation de réflexion affine $R'$ de $W(2p,2q,r)$ avec $G':=Im R'$, alors $G'$ est isomorphe à un quotient de $G$. Si $G=<s_{1},s_{2},s_{3}>$ des ``présentations'' de $G$ sont:
  \[
  (w(p,q,r),(y'^{2}s_{1})^{2}=(s_{1}y'^{2})^{2});
  \]
  \[
  (w(p,q,r),(x'^{2}s_{1})^{2}=(s_{1}x'^{2})^{2}).
  \]
\end{enumerate}
\end{proposition}
\begin{remark}
On a des résultats et formules analogues pour les groupes $W(2p,q,2r)$ et $W(p,2q,2r)$.
\end{remark}
\begin{proof}
Cela résulte de tout ce qui précède.
\end{proof}
\end{document}